\documentclass{amsart}
\usepackage{amsmath}
  \usepackage{paralist}
  \usepackage{graphics} 
  \usepackage{epsfig} 
\usepackage{graphicx}  \usepackage{epstopdf}
 \usepackage[colorlinks=true]{hyperref}
\hypersetup{urlcolor=blue, citecolor=red}

\newtheorem{theorem}{Theorem}[section]

\newtheorem{lemma}[theorem]{Lemma}

\theoremstyle{definition}

\newcommand{\R}{\mathbb{R}}
\newcommand{\bE}{\boldsymbol{E}}
\newcommand{\bA}{\boldsymbol{A}}
\newcommand{\Id}{\boldsymbol{I}}
\newcommand{\bG}{\boldsymbol{G}}
\title[Interaction estimates ] 
 {New interaction estimates for the Baiti-Jenssen system}

\author[Laura Caravenna and Laura V. Spinolo]{}


 \email{laura.caravenna@math.unipd.it}
 \email{spinolo@imati.cnr.it}

\begin{document}
\maketitle

\centerline{\scshape Laura Caravenna}
\medskip
{\footnotesize
 \centerline{Dipartimento di Matematica,}
 \centerline{Universit\`a degli Studi di Padova,}
 \centerline{ via Trieste 63, 35121 Padova, Italy}
} 

\medskip

\centerline{\scshape Laura V. Spinolo}
\medskip
{\footnotesize
 \centerline{IMATI-CNR,}
 \centerline{via Ferrata 1, 27100 Pavia, Italy}
}

\bigskip


\begin{abstract}
We establish new interaction estimates for a system introduced by Baiti and Jenssen. These estimates are pivotal to the analysis of the wave front-tracking approximation. In a companion paper we use them to construct a counter-ex\-am\-ple which shows that Schaeffer's Regularity Theorem for scalar conservation laws does not extend to systems. The counter-example we construct shows, furthermore, that a wave-pattern containing infinitely many shocks can be robust with respect to perturbations of the initial data. The proof of the interaction estimates is based on the explicit computation of the wave fan curves and on a perturbation argument. 
\end{abstract}

\section{Introduction}
We deal with the system of conservation laws 
\begin{equation}
\label{e:clbaj}
 \partial_t U + \partial_x \big[ F_\eta (U) \big] =0.
\end{equation}
The unknown $U=U(t, x)$ attains values in $\R^3$:
\begin{align*}
& U : &[0, + &\infty[ \times\R &&\to &&\R^{3}\\
&  &&(t,x) && \mapsto & U&=
\begin{pmatrix}u\\v\\w\end{pmatrix}
\end{align*}
and the flux function $F_\eta: \R^3 \to \R^3$ is defined as 
\begin{equation}
\label{E:pertSyst}
F_{\eta}(U) : = 
\left(
\begin{matrix}
	\displaystyle{ 4 \big[ (v-1) u - w \big] + \eta p_1(U) \phantom{\int}} \\
	v^2 \\
	\displaystyle{4 \Big\{ {v (v-2)u} - (v-1) w \Big\}
	 + \eta p_3(U) } \\
	\end{matrix}\right). 
\end{equation}	
In the previous expression, the parameter $\eta$ attains values in the interval $[0, 1/4[$
and {to simplify the exposition we fix} the functions $p_1$ and $p_3$ by setting 
\begin{align}
\label{E:pq} 
 p_1(U) 
 &=2 u w- 2 u^{2} (v-1 ), \\
 p_3(U) 
 &= w^2-u^2 (v-2) v .
\end{align}
Note, however, that for some of the results discussed in the following the precise expression of the functions $p_1$ and $p_3$ is irrelevant.

System~\eqref{e:clbaj},\eqref{E:pertSyst} was introduced by Baiti and Jens\-sen in~\cite{BaJ, Jenssen} and it was used to construct an example of a Cauchy problem where the initial data have finite, but large, total variation and the $L^{\infty}$-norm of the admissible solution blows up in finite time. 
More recently, the authors of the present paper used the Baiti-Jenssen system~\eqref{e:clbaj} to exhibit an explicit counter-example which shows that Schaeffer's regularity result for scalar conservation laws does not extend to systems, see~\cite{CaravennaSpinolo:counterex}.  The counter-example we construct shows, furthermore, that a wave-pattern containing infinitely many shocks can be robust with respect to perturbations of the initial data.  We refer to~\S~\ref{s:motivation} in the present paper for a brief overview of these counter-examples. See also~\cite{Caravenna:HYP}.

This note aims at establishing new quantitative interaction estimates for the Baiti-Jenssen systems~\eqref{e:clbaj},\eqref{E:pertSyst}. 
The estimates we obtain are pivotal to the analysis of the so-called wave front-tracking approximation of the Cauchy problem obtained by coupling~\eqref{e:clbaj} with an initial datum $U(0, \cdot) = U_0$. We refer to~\cite{Bre,Dafermos,HoldenRisebro} for an extended discussion on the wave front-tracking approximation. Here we only mention that the wave front-tracking algorithm is based on the construction of a piecewise constant approximation of the Cauchy problem. Under suitable conditions on the initial datum $U_0$ and on the flux function $F_\eta$, one can show that the wave front-tracking approximation converges to an \emph{admissible solution} of the Cauchy problem, see in particular the analysis in~\cite{Bre}. In~\cite{CaravennaSpinolo:counterex} we construct wave front-tracking approximations of the Cauchy problems obtained by coupling~\eqref{e:clbaj} with suitable initial data. We then rely on the wave front-tracking approximation to establish qualitative properties of the limit solutions. In the following we do not consider all the possible interactions one has to handle when constructing the wave front-tracking approximation. We only discuss those that we encounter in~\cite{CaravennaSpinolo:counterex} and that cannot be handled by relying on straightforward considerations on the structure of the flux $F_\eta$.

Before going into the technical details, we make some further remarks. First, in the present note we fix a very specific system in the wider class considered in~\cite{BaJ}.  The motivation for this choice is twofold:
i) it simplifies the notation and ii) the analysis in the present note is sufficient for the applications 
in~\cite{CaravennaSpinolo:counterex}. Note, moreover, that in the proof of Lemma~\ref{l:22:heuristic} we use (although not in an essential way) the exact 
expression of the function $F_\eta$ evaluated at $\eta=0$. However, we are confident that our results can be extended to wider classes of systems of the type considered in~\cite{BaJ}.

Second, in this note the only occurrences where we explicitly use the precise expression of the functions $p_1$ and $p_3$  is in the results discussed in~\S~\ref{ss:baj}. More precisely, the proofs of Lemmas~\ref{l:22:heuristic} and~\ref{l:12:heuristic} both rely on a perturbation argument: we first show that our system verifies the statement of the lemmas in the case $\eta=0$ and then we show that the same holds provided $\eta$ is sufficiently small. The proof of Lemma~\ref{l:12:heuristic} is completely independent of the specific expression of  $p_1$ and $p_3$. In the perturbation argument in the proof of Lemma~\ref{l:22:heuristic} we use some results from~\S~\ref{ss:baj}, but we never directly use the specific expression of $p_1$ and $p_2$.

Third, the Baiti-Jenssen~\eqref{e:clbaj} system in not physical, in the sense that it does not admit strictly convex entropies, see~\cite{BaJ} for a proof. It is natural to wonder whether or not the results established in the present note can be extended to physical systems. Very loosely speaking, by combining Lemmas~\ref{l:12:heuristic} and~\ref{l:22:heuristic} below with the analysis in~\cite[\S3.1-3.2]{CaravennaSpinolo:counterex} we get the following statement. Under suitable conditions, the only waves generated at the interactions between two shocks are shock waves, or, more precisely, no rarefaction waves are generated at the interaction between two shocks. There are actually several physical systems that share this property: for instance, one can consider the $2\times2$ example discussed by DiPerna in~\cite[\S 5]{DiPerna} and assume that the data have sufficiently small total variation. We refer to~\cite[\S 4]{BressanCoclite} for the analysis of shock interactions for this system. 
On the other hand, a much more challenging question is whether or not there is any physical system that 
exhibit the same behaviors as those discussed in~\cite{BaJ,CaravennaSpinolo:counterex}. In other words, one can wonder whether or not a physical system can i) exhibit finite time blow up  
or ii) violate the regularity prescribed, for scalar conservation laws, by Schaeffer's Theorem. To the best of the authors' knowledge, the answers to the above questions is presently open.

We now give some technical details about the estimates we establish. First, we point out that the Baiti-Jenssen system~\eqref{e:clbaj} is strictly hyperbolic in the unit ball, which amounts to say that the Jacobian matrix $DF_\eta$ admits three real and distinct eigenvalues 
\begin{equation}
\label{e:sh}
 \lambda_1 (U) < \lambda_2 (U) < \lambda_3 (U)
\end{equation}
for every $U$ such that $|U| <1$. Also, if $\eta >0$ every characteristic field is genuinely nonlinear. In other words, 
let $\vec r_1, \dots, \vec r_3$ denote the right smooth eigenvectors associated to the eigenvalues $\lambda_1, \lambda_2, \lambda_3$. Then 
\begin{equation}
\label{e:genuinenonlinear}
 \nabla \lambda_i (U) \cdot \vec r_i (U) \ge c > 0 
\end{equation}
for some suitable constant $c>0$ and for every $i=1, 2, 3$ and $|U| <1$. In the following, we distinguish three families of shocks: we term a given shock 1-, 2- or 3-shock depending on whether the speed of the shock is close to $\lambda_1$, $\lambda_2$ or $\lambda_3$. 

We also point out that establishing interaction estimates for system~\eqref{e:clbaj} boils down to the following. Consider the so-called Riemann problem, namely the Cauchy problem obtained by coupling~\eqref{e:clbaj} with an initial datum in the form 
\begin{equation}
\label{e:riepuelle}
 U(0, x) : =
 \left\{
 \begin{array}{ll}
 U_\ell & x <0 \\
 U_r & x >0, \\
 \end{array}
 \right.
\end{equation}
where $U_\ell$, $U_r \in \R^3$ are constant states. The above problem admits, in general, infinitely many distributional solutions: we term~\emph{admissible} the solution constructed by Lax in the pioneering work~\cite{Lax}, see \S~\ref{ss:lax} for a brief overview. Establishing interaction estimates for~\eqref{e:clbaj} amounts to establish estimates on the admissible solution of the Riemann problem~\eqref{e:clbaj}-\eqref{e:riepuelle} in the case when $U_\ell$ and $U_r$ satisfy suitable structural assumptions. 
\begin{figure}
\begin{center}
\caption{Left: interaction between two 2-shocks. Right: interaction between a 2-shock and a 1-shock.}
\label{F:2interazione}
\includegraphics[width=\linewidth]{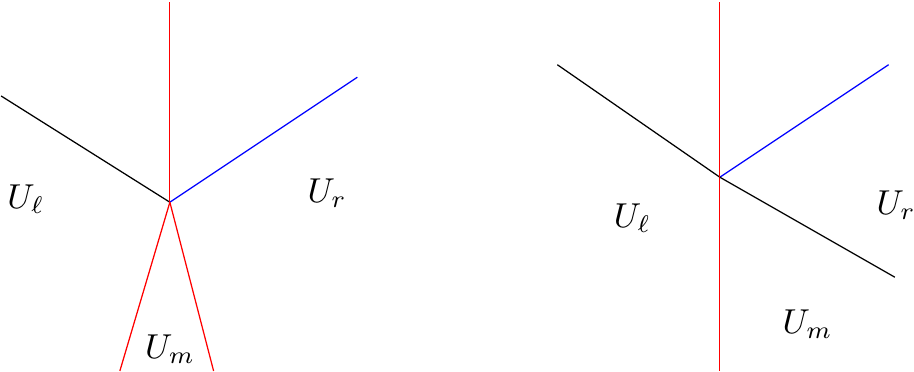}
\end{center}
\end{figure}

The first case we consider is the case of the interaction of two $2$-shocks, see Figure~\ref{F:2interazione}, left part. In other words, we assume that there is a state $U_m \in \R^3$ such that
\begin{itemize}
\item $U_\ell$ and $U_m$ are the left and the right states of a Lax admissible 2-shock, 
\item $U_m$ and $U_r$ are the left and the right states of a Lax admissible 2-shock and 
\item the shock between $U_\ell$ and $U_m$ has higher speed than the shock between $U_m$ and $U_r$. 
\end{itemize}
We now give an heuristic formulation of our interaction estimate and we refer to~\S~\ref{s:22} for the rigorous statement, which requires some technical notation.  Here we only point out that the \emph{strength} of a shock is a quantity defined in \S~\ref{ss:lax} which is proportional to the modulus of the difference between the left and the right state of the shock. 
\begin{lemma}
\label{l:22:heuristic} 
Fix a constant $a$ such that $0<a<1/2$ and set $U_{\sharp}: = (a, 0, -a)$. Consider the interaction between two 2-shocks and assume that the states $U_\ell$ and $U_r$ are sufficiently close to $U_{\sharp}$. If the strengths of the interacting 2-shocks are sufficiently small, then the admissible solution of the Riemann problem~\eqref{e:clbaj}-\eqref{e:riepuelle} is obtained by patching together {a 1-shock, a 2-shock and a 3-shock.} 
\end{lemma}
We remark that the relevant point in the above result is that the solution of the Riemann problem that we consider in the statement contains no rarefaction wave.

The second case we consider is the case of the interaction {between} a $1$-shock and a $2$-shock, see Figure~\ref{F:2interazione}, right part. In other words, we assume that there is a state $U_m \in \R^3$ such that 
\begin{itemize}
\item $U_\ell$ and $U_m$ are the left and the right states of a Lax admissible 2-shock, 
\item $U_m$ and $U_r$ are the left and the right states of a Lax admissible 1-shock.
\end{itemize}
The case of the interaction of a $3$-shock with a $2$-shock is analogous. 
We now give an heuristic formulation of our result and we refer to~\S~\ref{s:12} for the rigorous statement. 
\begin{lemma}
\label{l:12:heuristic}
Consider the interaction between a $1$-shock and a $2$-shock and assume both shocks have sufficiently small strength. Then the admissible solution of the Riemann problem~\eqref{e:clbaj}-\eqref{e:riepuelle} is obtained by patching together  {a 1-shock, a 2-shock and a 3-shock.}  Also, we establish quantitative bounds from above and from below on the strength of the outgoing shocks, see formulas~\eqref{e:sigmaprimotau}. 
\end{lemma}

Note that the fact that the three outgoing waves are shocks follows from the analysis in~\cite{BaJ}. Also, the bound from above on the strength of the outgoing $3$-shocks follows from by now classical interaction estimates, see~\cite[Page~133, (7.31)]{Bre}: the main novelty in Lemma~\ref{l:12:heuristic} is that we have a new bound from below on the strength of the outgoing $3$-shock, see the left hand side of  formula~\eqref{e:sigmaprimotau}. This estimate is important for the analysis in~\cite{CaravennaSpinolo:counterex}.

This note is organized as follows. In \S~\ref{s:overview} we go over some previous results. In particular, in \S~\ref{s:motivation} we provide some motivation for studying the Baiti-Jenssen system~\eqref{e:clbaj} by describing two counter-examples that use it. In \S~\ref{ss:lax} we recall some results from~\cite{Lax} and in~\S~\ref{ss:baj} we apply these results to the Baiti-Jenssen system. In~\S~\ref{s:22} we discuss the interaction of two $2$-shocks and in \S~\ref{s:12} the interaction of a $1$-shock and a $2$-shock. 

\section{Overview of previous results}
\label{s:overview}
 For the reader's convenience, in this section we go over some previous results. More precisely:
 \begin{itemize}
 \item[\S~\ref{s:motivation}:] we discuss two counter-examples based on the Baiti-Jenssen system~\eqref{e:clbaj}: the original one in~\cite{BaJ} and a more recent one devised in~\cite{CaravennaSpinolo:counterex}.
 \item[\S~\ref{ss:lax}:] we follow the famous work by Lax~\cite{Lax} and we outline the construction of the solution of the Riemann problem. 
 \item[\S~\ref{ss:baj}:] we apply  Lax's construction to the Baiti-Jenssen system. 
 \end{itemize}
 \subsection{Counter-examples based on the Baiti-Jens\-sen system}
\label{s:motivation}
This paragraph is organized as follows:
 \begin{itemize}
 \item[\S~\ref{sss:counterex1}:] we discuss the counter-example in~\cite{BaJ}
 \item[\S~\ref{sss:counterex2}:] we discuss the counter-example in~\cite{CaravennaSpinolo:counterex}. 
\end{itemize}

Before dealing with the specific examples, we recall two main features of the Baiti-Jenssen system: first, it is strictly hyperbolic, namely~\eqref{e:sh} holds. Note that strict hyperbolicity is a standard hypothesis for results concerning systems of conservation laws, see~\cite{Dafermos}. Also, if $\eta >0$ every characteristic field is genuinely nonlinear, which means that~\eqref{e:genuinenonlinear} is satisfied for every $i=1, 2, 3$. This is a remarkable property
because loosely speaking systems where all the characteristic field are genuinely nonlinear are usually better behaved than general systems. For instance, the celebrated decay estimate by Ole{\u\i}nik~\cite{Oleinik}, which applies to scalar conservation laws with convex fluxes, has been extended to systems of conservation laws where all the characteristic field are genuinely nonlinear, see for instance the works by Glimm and Lax~\cite{GlimmLax}, by Liu~\cite{Liu_decay} and, more recently, by Bressan and Co\-lom\-bo~\cite{BressanColombo}, Bressan and Goatin~\cite{BressanGoatin} and Bressan and Yang~\cite{BressanYang}, while for balance laws we refer to Christoforou and Trivisa~\cite{ChristoforouTrivisa}.

\subsubsection{Finite time blow up of admissible solutions with large total variation}
\label{sss:counterex1}
Consider the general system of conservation laws
\begin{equation}
\label{e:cl}
 \partial_t U + \partial_x \big[ F(U) \big] =0,
\end{equation}
where the unknown $U(t, x)$ attains values in $\R^N$, the variables $(t, x) \in [0, + \infty[\times\R$ and the flux function $F:\R^N \to \R^N$ is smooth and strictly hyperbolic~\eqref{e:sh}. Consider furthermore the Cauchy problem obtained by coupling~\eqref{e:cl} with the initial condition 
\begin{equation}
\label{e:id}
U(0, \cdot) = U_0.
\end{equation}
 Under some further technical assumption on the structure of the flux, Glimm~\cite{Glimm} established existence of a global in time solution of the Cauchy problem provided that $\mathrm{TotVar \,} U_0$, the total variation of the initial datum, is sufficiently small. Under the same assumptions, Bressan and several collaborators established uniqueness results, see~\cite{Bre} for a detailed exposition. 

The requirement that the total variation $\mathrm{TotVar \,} U_0$ is small is highly restrictive, but necessary to obtain well-posedness results unless additional assumptions are imposed on the flux function $F$. Indeed, explicit examples have been constructed of systems and data $U_0$ where $\mathrm{TotVar \,} U_0$ is finite, but large, and the admissible solution blows up in finite time. In particular, in~\cite{BaJ} Baiti and Jenssen constructed an initial datum for system~\eqref{e:clbaj} such that the $L^{\infty}$-norm of the admissible solution blows up in finite time. The solution is admissible in the sense that it is piecewise constant and every shock is Lax admissible. For further examples of finite time blow up, see the references in~\cite{BaJ} and~\cite{Dafermos}.

\subsubsection{Schaeffer's Regularity Theorem does not extend to systems}
\label{sss:counterex2}
In~\cite{Schaeffer} Schaeffer established a regularity result which can be loosely speaking formulated as follows. 
Consider a scalar conservation law with strictly convex flux, 
namely equation~\eqref{e:cl} in the case when $U(t, x)$ attains real values and $F : \R \to \R$ is uniformly convex, i.e. $F''\ge c >0$ for some constant $c>0$. The work by Kru{\v{z}}kov~\cite{Kru} establishes existence and uniqueness of the so-called \emph{entropy admissible solution} of the Cauchy problem posed by coupling~\eqref{e:cl} and~\eqref{e:id}. 
It is known that, even if $U_0$ is smooth, the entropy admissible solution can develop shocks,   namely discontinuities that propagate in the $(t,x)$-plane. Schaeffer's Theorem states that, for a generic smooth initial datum, the number of shocks of the entropy admissible solution is locally finite. The word ``generic" is here to be interpreted in a suitable technical sense, which is related to the Baire Category Theorem, see~\cite{Schaeffer} for the precise statement. 

In~\cite{CaravennaSpinolo:counterex} we discuss whether or not Schaeffer's Theorem extends to systems of conservation laws where every characteristic field is genuinely nonlinear, name\-ly~\eqref{e:genuinenonlinear} holds. Note that the assumption that every characteristic field is genuinely nonlinear can be loose\-ly speaking regarded as the analogous for systems of the condition (which applies to scalar equations) that the flux is strictly convex. Indeed, regularity results for scalar equations with strictly convex fluxes have been extended to systems where every characteristic field is genuinely nonlinear: as we mentioned before, this is the case of Ole{\u\i}nik's~\cite{Oleinik} decay estimate, see for instance~\cite{BressanColombo,BressanGoatin,BressanYang,ChristoforouTrivisa,GlimmLax,Liu_decay} for possible extensions to systems. Also, the $SBV$ regularity result by Ambrosio and De Lellis~\cite{AmbrosioDeLellis}, which applies to scalar conservation laws with strictly convex fluxes, has been extended to systems where every characteristic field is genuinely nonlinear, see~\cite{AnconaNguyen,BiaCar,Dafermos_sbv}.

Despite the above considerations, in~\cite{CaravennaSpinolo:counterex} we exhibit an explicit example which rules out the possibility of extending Schaeffer's Theorem to systems of conservation laws where every characteristic field is genuinely nonlinear. More precisely, we construct a ``big" set of initial data such that the corresponding solutions of the Cauchy problems for the Baiti-Jenssen system~\eqref{e:clbaj} develop infinitely many shocks on a given compact set of the $(t, x)$-plane. The term 
``big" is to be again interpreted in a suitable technical sense, which is related to the Baire Category Theorem, see~\cite{CaravennaSpinolo:counterex} for the technical details.
\subsection{The Lax solution of the Riemann problem}
\label{ss:lax}
We consider a system of conservation laws~\eqref{e:cl} and we assume that $F: \R^3 \to \R^3$ is strictly hyperbolic~\eqref{e:sh} and that every characteristic field is genuinely nonlinear, namely~\eqref{e:genuinenonlinear} holds for $i=1,2, 3$.  Lemma~\ref{l:eige} below states that the Baiti-Jenssen system satisfies these conditions.
The Riemann problem is posed by coupling~\eqref{e:cl} with an initial datum in the form
\begin{equation}
\label{e:riep}
 U(0, x) : =
 \left\{
 \begin{array}{ll}
 U^- & x <0 \\
 U^+ & x >0, \\
 \end{array}
 \right.
\end{equation}
where $U^+$ and $U^-$ are given states in $\R^3$. In~\cite{Lax}, Lax constructed a solution of the Riemann problem~\eqref{e:cl}-\eqref{e:riep} under the assumptions that the states $U^+$ and $U^-$ are sufficiently close: we now briefly recall the key steps of the analysis in~\cite{Lax}.

We fix $i=1, 2, 3$ and $\bar U \in \R^3$ and we define the \emph{$i$-wave fan curve} through $\bar U$ by setting 
\begin{equation}
\label{e:wavefan}
 D_i [s, \bar U] : = 
 \left\{
 \begin{array}{ll}
 R_i [s, \bar U] & s \ge 0 \\
 S_i [s, \bar U] & s < 0 \\
 \end{array}
 \right.
\end{equation} 
In the previous expression, $R_i$ is the \emph{$i$-rarefaction curve} through $\bar U$ and $S_i$ is the \emph{$i$-Hugoniot locus} through $\bar U$. The $i$- rarefaction curve $R_i$ is the integral curve of the vector field $\vec r_i$, namely the solution of the Cauchy problem
\begin{equation}
\label{e:rarefaction}
 \left\{
 \begin{array}{ll}
 \displaystyle{\frac{d R_i}{ds }
 = \vec r_i \big( R_i \big)} \\
 \phantom{ciao} \\
 R_i [0, \bar U] = \bar U. \\
 \end{array}
 \right.
\end{equation} 
The $i$-th Hugoniot locus $S_i$ is the set of states that can be joined to $\bar U$ by a shock with speed close to $\lambda_i(\bar U)$. The $i$-Hugoniot locus $S_i$ is determined by imposing the Rankine-Hugoniot conditions.  We term the value $|s_i|$ \emph{strength} of the $i$-wave connecting the states $\bar U$ (on the left) and $D_i [s, \bar U]$ (on the right). Note that, owing to~\eqref{e:wavefan}, when $s_i>0$ the $i$-wave is a $i$-th rarefaction wave, when 
$s_i<0$ the $i$-wave is an $i$-shock satisfying the so-called \emph{Lax admissibility criterion}.  
The solution of the Riemann problem~\eqref{e:cl}-\eqref{e:riep} is computed by imposing 
$$
 U^+ = D_3 \Big[ s_3, D_2 \big[ s_2, D_1[s_1, U^- ] \big] \Big]
$$
and by using the Local Invertibility Theorem to solve for $(s_1, s_2, s_3)$. 
From the value of $(s_1, s_2, s_3)$ one can reconstruct a solution of the Riemann problem~\eqref{e:cl}-\eqref{e:riep}, see~\cite{Lax} for the precise construction. This solution is obtained by patching together rarefaction waves and shocks that satisfy the Lax admissibility criterion. In the following, we refer to this solution as the \emph{Lax solution} of the Riemann problem~\eqref{e:cl}-\eqref{e:riep}. 

\subsection{The wave fan curves of the Baiti-Jenssen system}
\label{ss:baj}
We collect in this paragraph some features of the Baiti-Jenssen system. For the proof, we refer to~\cite{BaJ,CaravennaSpinolo:counterex}. 

The first result states that in the unit ball the Baiti-Jenssen system 
is strictly hyperbolic whenever $0 \leq \eta < 1/4$. Also, when $\eta >0$ all the characteristic fields are genuinely nonlinear. Note that when $\eta =0$ this last condition is lost because two characteristic fields became linearly degenerate. See~\cite{BaJ} or~\cite{CaravennaSpinolo:counterex} for the explicit computations. 
\begin{lemma}
\label{l:eige}
Assume that $0 \leq \eta<1/4$ and that $U$ varies in the unit ball, {$|U|~<~1$}. Then the Baiti-Jenssen system with flux~\eqref{E:pertSyst} is strictly hyperbolic, namely~\eqref{e:sh} holds true. If we also have $\eta >0$ then every characteristic field is genuinely nonlinear, name\-ly~\eqref{e:genuinenonlinear} is satisfied for 
$i=1,2,3.$ 
\end{lemma}
We now discuss the structure of the wave fan curves. We start by 
giving the explicit expression of the 1- and the 3-wave fan curve. 
 In the statement of the following result, we denote 
by $(\bar u, \bar v, \bar w)$ the components of the state $\bar U \in \R^3$. 
\begin{lemma}
\label{l:wavefan}
Consider the flux function~\eqref{E:pertSyst}, assume that $0< \eta <1/4$ and fix $\bar U \in \R^3$ such that 
$|\bar U| <1$. Then the following properties 
hold true. 
\begin{itemize}
\item[i)] The 1-wave fan curve $D_1 [\sigma, \bar U]$ is a straight line in the 
plane $v = \bar v$, more precisely
\begin{align}
\label{e:wavefan1}
 D_1 [\sigma, \bar U] &= \bar U + \sigma \vec r_1 (\bar U), 
 \\\notag
 & \text{where} \quad 
 \vec r_1 (\bar U) 
 = 
 \left(
 \begin{array}{cc}
 1\\
 0\\
 \bar v\\
 \end{array}
 \right). 
\end{align}
Note that $\vec r_1 (\bar U)$ is the first eigenvector of the Jacobian
matrix $DF (\bar U)$. Also, the states $\bar U$ (on the left)
and $D_1 [\sigma, \bar U] $ (on the right) are connected by a wave which is
\begin{itemize}
\item a 1-rarefaction wave when $\sigma >0$,
\item a Lax admissible 1-shock when $\sigma <0$.
\end{itemize}
\item[ii)] The 3-wave fan curve $D_3 [\tau, \bar U]$ is a straight line in the 
plane $v = \bar v$, more precisely
\begin{align}
\label{e:wavefan2}
 D_3 [\tau, \bar U] &= \bar U + \tau \vec r_3 (\bar U), 
 \\\notag
 & \text{where} \quad 
 \vec r_3 (\bar U) 
 = 
 \left(
 \begin{array}{cc}
 1\\
 0\\
 \bar v- 2 \\
 \end{array}
 \right). 
\end{align}
The vector $\vec r_3 (\bar U)$ is the third eigenvector of the Jacobian
matrix $DF (\bar U)$. Also, the states $\bar U$ (on the left)
and {$D_3 [\tau, \bar U] $} (on the right) are connected by a wave which is
\begin{itemize}
\item a 3-rarefaction wave when $\tau <0$,
\item a Lax admissible 3-shock when $\tau >0$.
\end{itemize}
\end{itemize}
\end{lemma}
Note that, for the 3-wave fan curve, the \emph{positive} values of $\tau$ correspond to shocks, the \emph{negative} values to rarefaction waves. This is the contrary with respect to~\eqref{e:wavefan} and it is a consequence of the fact that we use the same notation as in~\cite{BaJ,CaravennaSpinolo:counterex} and we choose the orientation of $\vec r_3$ in such a way that when $\eta>0$ condition~\eqref{e:genuinenonlinear} is replaced by the opposite inequality
$$
 \nabla \lambda_3 \cdot \vec r_3 < 0. 
$$
We now turn to the structure of the 2-wave fan curve. In the following statement, we use the notation 
\[
 U^{-}=
\begin{pmatrix}u^{-}\\v^{-}\\w^{-}\end{pmatrix}, \qquad
 U^{+}=
\begin{pmatrix}u^{+}\\v^{+}\\w^{+}\end{pmatrix}.
\]
Also, we consider \emph{entropy admissible solutions} of scalar conservation laws, in the Kru{\v{z}}kov~\cite{Kru} sense. 
\begin{lemma}
\label{l:wavefan2}
Assume that $U$ is a Lax solution of the Riemann problem~\eqref{e:cl}-\eqref{e:riep}. Then the second component $v$ is an entropy admissible solution of the Cauchy problem
\begin{equation}
\label{e:scariep}
\left\{
\begin{array}{lll}
\partial_t v + \partial_x [v^2] =0 \\
v(0, x) =
\left\{
\begin{array}{ll}
v^- & x <0 \\
v^+ & x >0. \\
\end{array}
\right.
\end{array}
\right.
\end{equation}
Also, we can choose the eigenvector $\vec r_{2}$ and the parametrization of the 2-wave fan curve $D_2 [s, \bar U]$ in such a way that the second component of $D_2 [s, \bar U]$ is exactly $\bar v +s$. 
\end{lemma}

\section{Interaction of two 2-shocks}
\label{s:22}
We first rigorously state Lemma~\ref{l:22:heuristic}
\begin{lemma}
\label{l:22}
There is a sufficiently small constant $\varepsilon >0$ such that the following holds. Fix a constant $a$ such that $0<a<1/2$ and set $U_{\sharp}: =(a, 0, - a)$. Assume that 
\begin{align*}
 & |U_\ell - U_{\sharp} | \leq \varepsilon a, && 0 \leq \eta \leq \varepsilon a ,\\
 &s_1, s_2 <0, && s_1,s_2 \in [-\varepsilon a ,0]. 
\end{align*}
 Assume furthermore that 
 \begin{equation}
 \label{e:primainterazione}
 U_r = D_2 \Big[s_2, D_2 [s_1, U_\ell] \Big]. 
 \end{equation}
 Then there are $\sigma <0$ and $\tau >0$ such that
 \begin{equation}
 \label{E:dini}
 U_r = D_3 \Big[ \tau, D_2 \big[ s_1 + s_2, D_1 [\sigma, U_\ell] \big] \Big]\,.
 \end{equation}
\end{lemma}
Note that by combining~\eqref{E:dini} with the inequalities $\sigma <0$, $\tau >0$ and $s_1 + s_2 < 0$ we get that the three outgoing waves are all shocks. 
The proof of Lemma~\ref{l:22} is organized as follows:
\begin{itemize}
\item[\S~\ref{ss:22etanonzero}:] by relying on a perturbation argument, we show that 
the proof of Lemma~\ref{l:22} boils down to the proof of the Taylor expansion~\eqref{E:mclaurin}. 
\item[\S~\ref{ss:22etazero}:] we complete the proof by establishing~\eqref{E:mclaurin}. 
\end{itemize}
\subsection{Proof of Lemma~\ref{l:22}: first step}
\label{ss:22etanonzero}
We start with some preliminary considerations. Assume that the states $U_\ell$ and $U_r$ satisfy~\eqref{e:primainterazione}. Next, solve the Riemann problem between $U_\ell$ (on the left) and $U_r$ (on the right): owing to~\cite{Lax}, this amounts to determine by relying on the Local Invertibility Theorem the real numbers $\sigma$, $s$ and $\tau$ such that 
\begin{equation}
 \label{E:dini2}
 U_r = D_3 \Big[ \tau, D_2 \big[ s, D_1 [\sigma, U_\ell] \big] \Big].
 \end{equation}
Establishing the proof of Lemma~\ref{l:22} amounts to prove that $s = s_1 + s_2 <0$ 
and that $\sigma <0$, $\tau>0$. 

To prove that $s = s_1 + s_2$ we recall Lemma~\ref{l:wavefan2} and the fact that the $v$ component is constant along the 1-st and the 3-rd wave fan curves $D_1$ and $D_3$. We conclude that $s = v_{r} - v_{\ell}= s_1 + s_2 <0$ . Note that $v_r$ and $v_{\ell}$ are the second component of $U_r$ and $U_{\ell}$. 

We are left to prove that $\sigma<0$ and $\tau>0$. We first introduce some notation: we regard $\sigma$ and $\tau$ as functions of $\eta$, $s_1$ and $s_2$ and $U_\ell$ and we write $\sigma_\eta (s_1, s_2, U_\ell)$ and $\tau_\eta (s_1, s_2, U_\ell)$ to express this dependence. Note that $\sigma$ and $\tau$ depend on $\eta$ because the wave fan curve $D_2$ depends on $\eta$. 

Owing to the Implicit Function Theorem, the regularity of $\sigma_\eta (s_1, s_2, U_\ell)$ and $\sigma_\eta (s_1, s_2, U_\ell)$ is at least the same as the regularity of the functions $D_1$, $D_2$ and $D_3$. Also, note that the Lax Theorem~\cite{Lax} (see also~\cite[p.101]{Bre}) states that the wave fan curves $D_1$, $D_2$ and $D_3$ are $C^2$.  The reason why we can achieve $C^\infty$ regularity is because we are actually considering the wave fan curves in regions where they are $C^\infty$. To see this, we first point out that, owing to~\eqref{e:wavefan1} and~\eqref{e:wavefan2}, 
 the wave fan curves $D_1$, $D_3$ are straight lines and hence they are $C^\infty$. 
Next, we point out that we are only interested in \emph{negative} values of $s_1+s_2$. Hence, we can replace the 2-wave fan curve $D_2$ defined as in~\eqref{e:wavefan} with the 2-Hugoniot locus $S_2$. We recall that the 2-Hugoniot locus $S_2[s, \bar U]$ contains all the states that can be connected to $\bar U$ by a shock, namely all the states such that the couple $( \bar U, S_2[s, \bar U]) $ satisfies the Rankine-Hugoniot conditions. The 2-Hugoniot locus $S_2[s, \bar U]$ is $C^\infty$ and by combining all the previous observations we can conclude that $\sigma_\eta (s_1, s_2, U_\ell)$ and $\tau_\eta (s_1, s_2, U_\ell)$ are both $C^\infty$  with respect to the variables $(\eta, s_1, s_2, U_\ell)$. 

Next, we discuss the partial derivatives of $\sigma_\eta (s_1, s_2, U_\ell)$ and  
$\tau_\eta (s_1, s_2, U_\ell)$ with respect to $(s_1, s_2)$ at the point  $(\eta, 0, 0, U_\ell)$. 
By arguing as in the proof of estimate (7.32) in~\cite[p.133]{Bre} we conclude that  
\begin{itemize}
\item for every $U_\ell$, for every $\eta>0$ and every integer $k \ge 1$ we have the following equalities:
\begin{align}
\label{E:vanish}
 \left. \frac{\partial^k \sigma_\eta}{\partial s_1^k} \right|_{ (0, 0, U_\ell)}&=
 \left. \frac{\partial^k \sigma_\eta}{\partial s_2^k} \right|_{ (0, 0, U_\ell)}=
 \left. \frac{\partial^k \tau_\eta}{\partial s_1^k} \right|_{ (0, 0, U_\ell)}
 =
 \left. \frac{\partial^k \tau_\eta}{\partial s_2^k} \right|_{ (0, 0, U_\ell)}= 0.
\end{align}
\item For every $U_\ell$ and for every $\eta>0$ we also have the following equality concerning the derivatives of second order: 
$$
    \left. \frac{\partial^2 \sigma_\eta}{\partial s_1 \partial s_2} \right|_{ (0, 0, U_\ell)}=
     \left. \frac{\partial^2 \tau_\eta}{\partial s_1 \partial s_2} \right|_{ (0, 0, U_\ell)} =0. 
$$ 
\end{itemize} 
This implies that $\sigma_\eta$ and $\tau_\eta$ admit the following Taylor expansions
\begin{equation}
\label{e:taylor2}
\begin{split}
     &
     \sigma_\eta (s_1, s_2, U_\ell) = {\frac{1}{2}}
     \frac{\partial^3\sigma_{\eta}(0,0,U_{\ell})}{\partial s_{1}^2  \partial s_{2}} \, s_1^2 s_2 +
     {\frac{1}{2}} \frac{\partial^3\sigma_{\eta}(0,0,U_{\ell})}{\partial s_{1}\partial s_{2}^2 } \, s_1 s_2^2 \\ & 
     \qquad \qquad \qquad \quad+
      o (| (s_1, s_2) |) \, s_1 s_2 (s_1+ s_2)  \\ &
      \tau_\eta (s_1, s_2, U_\ell) = {\frac{1}{2}}
     \frac{\partial^3\tau_{\eta}(0,0,U_{\ell})}{\partial s_{1}^2 \partial s_{2}} \, s_1^2 s_2 +
    {\frac{1}{2}}
     \frac{\partial^3\tau_{\eta}(0,0,U_{\ell})}{\partial s_{1}\partial s_{2}^2 } \, s_1 s_2^2 +\\ & 
     \qquad \qquad \qquad \quad+
      o (| (s_1, s_2) |) \, s_1 s_2 (s_1+ s_2)\\      \end{split}
\end{equation}  
In \S~\ref{ss:22etazero} we prove that when $\eta=0$ and $U_\ell = U_{\sharp}$ the functions $\sigma$ and $\tau$ admit the Taylor expansions
\begin{subequations}
\label{E:mclaurin}
\begin{align}
\begin{split}
\begin{pmatrix} \sigma_0(s_1, s_2, U_{\sharp}) \\ \tau_0(s_1, s_2, U_{\sharp})\end{pmatrix}= \frac{a}{32} \begin{pmatrix} 1 \\ -1\end{pmatrix}&
s_1 s_2 (s_1+ s_2)
+ o (| (s_1, s_2) |^3). \end{split}
\end{align}
\end{subequations}
Next, we use the Lipschitz continuous dependence of the derivatives of 
third order with respect to $\eta$ and $U_\ell$ and we conclude that 
\begin{align*}
&\left|\frac{\partial^3\sigma_{\eta}(0,0,U_{\ell})}{\partial s_{1}^2 \partial s_{2}}-\frac{a}{{16}}\right|+\left|\frac{\partial^3\sigma_{\eta}(0,0,U_{\ell})}{\partial s_{1}\partial s_{2}^2 }-\frac{a}{{16}}\right|<C \varepsilon a
\\
&\left|\frac{\partial^3\tau_{\eta}(0,0,U_{\ell})}{\partial s_{1}^2 \partial s_{2}}+\frac{a}{{16}}\right|+\left|\frac{\partial^3\tau_{\eta}(0,0,U_{\ell})}{\partial s_{1}\partial s_{2}^2}+\frac{a}{{16}}\right|< C \varepsilon a
\end{align*}
provided that $0 \leq \eta \leq \varepsilon a$ and $|U_\ell - U_\sharp| \leq \varepsilon a$. In the above expression, $C$ denotes a universal constant. By plugging the above expressions into~\eqref{e:taylor2} and recalling that $s_1, s_2 < 0$ we can eventually conclude that, if $\varepsilon$ is sufficiently small,  then   
\begin{align*}
 \sigma_\eta(s_1, s_2, U_\ell)<  \phantom{-}\frac{a}{64} s_1 s_2 (s_1+ s_2)<0, \\
 \tau_\eta(s_1, s_2, U_\ell)> -\frac{a}{64 }s_1 s_2 (s_1+ s_2)>0.
\end{align*}
The proof of the lemma is complete.  

\subsection{Proof of formula~\eqref{E:mclaurin}}
\label{ss:22etazero}
The proof of the Taylor expansion~\eqref{E:mclaurin} is divided into two parts:
\begin{itemize}
\item[\S~\ref{sss:2H}:] as a preliminary result we determine the structure of the Hugoniot locus $S_2[s, \underline U]$
\item[\S~\ref{sss:conclu}:] we conclude the proof. 
\end{itemize}
Note that in this paragraph we always assume $\eta =0$ because formula~\eqref{E:mclaurin} deals with this case. 
\subsubsection{The 2-Hugoniot locus}
\label{sss:2H}
Before giving the technical results, we introduce some notation. First, we recall that we term $F_0$ the flux function $F_\eta$ in~\eqref{E:pertSyst} in the case when $\eta=0$. 
In the following, we will mostly focus on the behavior of the first and the third component of $U$. Hence, it is convenient to term $\hat U$ and $\hat F_0$ the vectors obtained by erasing the second components of $U$ and $F_0$, respectively. We have the relation 
 \begin{equation}
\label{E:erratacorrige}
\widehat F_{0}(U)  =
	 4	\left(\begin{matrix}
	v-1 & -1\\
	v(v-2) & 1-v
	\end{matrix}\right) 
	\left(\begin{matrix} u\\ w	\end{matrix}\right)  = \widehat J(v) \cdot \widehat U,
\end{equation}
where we have also introduced the $2 \times 2$ matrix $\widehat J(v)$.

Finally, we recall that we term $S_2 [s, \bar U]$ the 2-Hu\-go\-niot locus passing through $\bar U$, namely the set of states that can be connected to $\bar U$ by a (possibly not admissible) shock of the second family. Also, as usual we denote by $\bar u$, $\bar v$ and $\bar w$ the first, second and third component of $\bar U$, respectively. We use the notation $\widehat{\bar U}=(\bar u,\bar w)$.
\label{}
\begin{lemma}
\label{L:twoHugoniotlocus}
Fix $\eta =0$ and assume that $|2 \bar v + s| < 4$, then the 
2-Hugoniot locus through $\bar U$ has the following expression:
the second component of $S_{2}[s, \bar U]$ is $\bar v+s$ while the first and third components are
\begin{equation}
\label{E:explicit}
\begin{split}
&\widehat{S_{2}}[s, \bar U]= 
\widehat{\bar U} + \bE(\bar v, s) \widehat{\bar U}
\end{split}
\end{equation}
where the $2 \times 2$ matrix $\bE (\bar v, s)$ is 
\begin{align*}
\bE(\bar v,s)
=
\frac{4s}{(2 \bar v + s)^2 - 16} 
\begin{pmatrix}
s+4-2 \bar v& 4\\
({s}+4) (s-2) +4 \bar v& 3s-4+2 \bar v \end{pmatrix} .
\end{align*} 
\end{lemma}
\begin{proof}
By Lemma~\ref{l:wavefan2} the second component of $S_{2}[s,\bar U]$ is $\bar v+s$.
To construct $S_2 [s, \bar U]$ we use the Rankine-Hugoniot conditions, which are a system of 3 equations. Owing to Lemma~\ref{l:wavefan2}, the second equation reads  
\[
\gamma  s=(\bar v+s)^{2}- \bar v^{2}
\]
and this implies that the speed  $\gamma$ of the 2-shock is 
\begin{equation}
\label{E:gamma}
\gamma  =2 \bar v+s.
\end{equation}
We define the vector $\mathfrak A(s,\bar U)$ by setting
\[
\mathfrak A(s,\bar U)
:=\widehat{S_{2}}[s, \bar U]-\widehat{\bar U}
\]
and we point out that to establish Lemma~\ref{L:twoHugoniotlocus} we are left to show that 
\begin{equation}
\label{e:correzione}
      \mathfrak A(s,\bar U)={ \bE (\bar v, s)} \widehat{\bar U}.
\end{equation}
The first and the third equations in the Rankine-Hugoniot conditions can be written as
\begin{equation}
\label{e:rh13}
\begin{split}
\gamma \mathfrak A(s,\bar U)
=
\widehat{J}(\bar v+s) \left[\widehat{\bar U}+\mathfrak A(s,\bar U)\right]-\widehat{J}(\bar v)\widehat{\bar U}, 
\end{split}
\end{equation}
where $\widehat J$ {is the same as in~\eqref{E:erratacorrige}.}
Next, we introduce  the $2\times2$ matrix
\[
\begin{split}
\bA(v,\gamma) &=\gamma\Id - \widehat{J}(v)\\
&=	\left(\begin{matrix}
	\gamma & 0\\
	0 & \gamma
	\end{matrix}\right) - 4	\left(\begin{matrix}
	v-1 & -1\\
	v(v-2) & 1-v
	\end{matrix}\right) ,
\end{split}
\]
and we rewrite~\eqref{e:rh13} as 
\[
\bA^{}\left(\bar v +s,\gamma \right)
\mathfrak A(s,\bar U) =
 \left[ \widehat J(\bar v+s)-\widehat J(\bar v)
 \right] \widehat{\bar U} ,
\]
{which implies~\eqref{e:correzione}} provided that
\[
\begin{split}
\bE(\bar v,s)&=\bA^{-1}(\bar v +s,\gamma) \big[ \widehat J(\bar v+s)-\widehat J(\bar v)
 \big] \\
\end{split}
\]
By recalling that $\gamma=2\bar v+s$  we can compute the explicit expression 
of the above matrices: 
\begin{align*}
&\bA(\bar v+s,2 \bar v+ s)= 
\begin{pmatrix}
4 -3s - 2 \bar v & 4\\
 4 (2 - \bar v-s) (\bar v+s) &   6 \bar v+5s  -4 
\end{pmatrix},
\\
&\widehat J(\bar v+s)-\widehat J(\bar v)=4s\begin{pmatrix}
1 & 0\\
 s+2\bar v -2 & -1
\end{pmatrix} .
\end{align*}
The determinant of the matrix $\bA(\bar v+s,2 \bar v+ s)$ is 
\[
\underline{\mathrm{det}}:=(2 \bar v+s)^2 -16
\]
and hence the matrix is invertible when $|2 \bar v +s| <4$. 
We can now complete the lemma by computing the explicit expression 
of $\bE$, namely 
\begin{align*}
&\bE(\bar v,s)
=
\frac{1}{\underline{\mathrm{det}}}
\begin{pmatrix} 6\bar v + 5s-4 & -4 \\ 
4 (\bar v+s-2) (\bar v+s) & 4-3 s - 2 \bar v \end{pmatrix}
\cdot
4s\begin{pmatrix}
1 & 0\\
 s+2\bar v -2 & -1
\end{pmatrix} \\
&=
\frac{4s}{\underline{\mathrm{det}}}
\begin{pmatrix}
s+4-2\bar v& 4\\
({s}+4) (s-2) +4\bar v& 3s-4+2\bar v \end{pmatrix} 
.\qedhere
\end{align*} 
\end{proof}

\subsubsection{Conclusion of the proof of formula~\eqref{E:mclaurin}}
\label{sss:conclu}
We are now ready to establish~\eqref{E:mclaurin}. We first recall some notation: we consider the system of conservation laws with flux $F_0$, see~\eqref{E:pertSyst}. We consider the collision between two 2-shocks and we assume that $U_{\sharp}=(a, 0, -a)$, $U_m$ and $U_r$ are the left, middle and right states before the interaction. This means that for some $s_1<0$, $s_2<0$ we have 
\begin{equation}
\label{e:quasi}
\begin{split}
 U_r &= D_2 [s_2, U_m]= D_2 \big[ s_2, D_2 [s_1, U_{\sharp} ] \big] \\ &= S_2 \big[ s_2, S_2 [s_1, U_{\sharp}]   \big].
 \end{split}
\end{equation}
In the above expression, $S_2$ represents the $2$-Hugoniot locus. 
To establish the last equality we used the fact that $s_1$ and $s_2$ are both negative. We plug~\eqref{E:explicit} into~\eqref{e:quasi} and we use the equality $v_{\sharp}=0$: we arrive at
\begin{equation}
\label{E:prima}
\begin{split}
 &\widehat{ U_r }
 = \left[\widehat{U_{\sharp}} + \bE(0, s_1) \widehat{U_{\sharp}} \right]+ \bE(s_1, s_2) 
 \Big[ \widehat{U_{\sharp}} + \bE(0, s_1) \widehat{U_{\sharp}} \Big] \\
 & = 
\widehat{U_{\sharp}} + \Big[ \bE(0, s_1) + \bE(s_1, s_2) + \bE(s_1, s_2) \bE(0, s_1) \Big] \widehat{U_{\sharp}}.
\end{split}
\end{equation}
Next, we focus on the states after the interaction. By arguing as at the beginning of \S~\ref{ss:22etanonzero}, we conclude that it suffices to determine $\sigma= \sigma_0(s_1, s_2, U_{\sharp})$ and 
 $\tau= \tau_0(s_1, s_2, U_{\sharp})$ such that 
 $$
 U_r = D_3 \Big[ \tau, D_2 \big[ s_1 + s_2, D_1 [\sigma, U_{\sharp}] \big]\Big].
 $$
By the explicit expression of $D_1$ and $D_3$ and by applying {Lemma~\ref{L:twoHugoniotlocus}} we infer that the above equality implies 
\begin{equation}
\label{E:dopo}
 \begin{split}
 \widehat U_r & 
 = \left[\widehat{U_{\sharp}} + \sigma \widehat{\vec r_1} (0) \right] + \bE(0, s_1 + s_2)
 \Big[ \widehat{U_{\sharp}} + \sigma\widehat{\vec r_1}(0)\Big]
  + \tau \widehat{\vec r_3}(s_1 + s_2) \\
 & = \widehat{U_{\sharp}} + \bE(0, s_1 + s_2) \widehat{U_{\sharp}}
  +
 \Big[ \Id + \bE(0, s_1 + s_2) \Big] \sigma \widehat{\vec r_1}(0) + \tau \widehat{\vec r_3}(s_1 + s_2) \\
 & = \widehat{U_{\sharp}} + \bE(0, s_1 + s_2) \widehat{U_{\sharp}} +
 \mathbf{H} (s_1 + s_2)
 \left(
 \begin{array}{cc}
 \sigma \\
 \tau \\
 \end{array}
 \right) .
\end{split} 
\end{equation}
In the previous expression we denote by $\widehat{\vec r_1} $ and $\widehat{\vec r_3} $ the  vectors obtained from $\vec r_1$ and $\vec r_2$ by erasing the second component. Also, we introduced the matrix $\mathbf H$: its first column is 
$ \big[ \Id + \bE(0, s_1 + s_2) \big] {\widehat{\vec r_1} (0)}$, the second column is $\widehat{\vec r_3}(s_1 + s_2) $. In the following, we will prove that $\mathbf H(s_1+s_2)$ is invertible provided that $s_1$ and $s_2$ are both sufficiently close to $0$. 
By comparing~\eqref{E:prima} and~\eqref{E:dopo} we then obtain 
\begin{equation}
\label{E:quasifine}
 \left(
 \begin{array}{cc}
 \sigma \\
 \tau \\
 \end{array}
 \right) \\
 =  \underbrace{ \mathbf{H}^{-1} (s_1 + s_2)
 \Big[ \bE(0, s_1) + \bE(s_1, s_2) + \bE(s_1, s_2) \bE(0, s_1) 
 - \bE(0, s_1 + s_2) \Big]}_{\displaystyle{\bG(s_1, s_2)}} \widehat{U_{\sharp}}.
\end{equation}
 Assume that we have established the following asymptotic expansion for $\bG$:
\begin{equation}
\label{E:leftoprove}
 \bG(s_1, s_2) = 
 \frac{1}{32}
\left(\begin{array}{cc}
 4 & 3 \\ 2 & 3
\end{array}
 \right)
 s_{1}s_{2}(s_{1}+s_{2})
 +
 o(|{(s_{1},s_{2})}|^{3}).
\end{equation}
Then by plugging both~\eqref{E:leftoprove} and $ \widehat{U_{\sharp}} = (a, -a)$ into~\eqref{E:quasifine} we obtain the asymptotic expansion~\eqref{E:mclaurin}. 
Hence,  to conclude the proof of~\eqref{E:mclaurin} we are left to establish~\eqref{E:leftoprove}. 

First, we point out that, owing to the expression of $\bE$ in the statement of Lemma~\ref{L:twoHugoniotlocus}, 
\begin{align*}
\bE(0,s)
=&
\frac{4s}{s^2 - 16}
\left(\begin{array}{cc}
s+4& 4\\
({s}+4) (s-2) & 3s-4 \end{array}\right) .
\end{align*} 
This implies that when $s_{1}=s_{2}=0$, the matrix $\bE(0,s_{1}+s_{2})$ vanishes and hence 
\begin{align*}
\mathbf{H}^{-1}(0)& =\Big( \hat {\vec r}_{1}({0})|\hat {\vec r}_{3}({0}) \Big)^{-1}
\stackrel{\eqref{e:wavefan1}, \eqref{e:wavefan2}}{=}\left(\begin{array}{cc} 1 & 1 \\ 0 & -2\end{array}\right)^{-1}
=
\left(\begin{array}{cc}
 1 & 1/2 \\ 0 & -1/2
\end{array}
 \right).
\end{align*}
We compute now the asymptotic expansion of 
\[
\begin{split}
\bE({0},s_{1})
+
\bE({0}+s_{1},s_{2})
&-
\bE({0},s_{1}+s_{2})
+
\bE({0}+s_{1}, s_{2})\bE({0},s_{1}).
\end{split}
\]
By directly computing the sum of the above matrices, we obtain that we can factor the term
\[
\frac{4s_{1}s_{2}(s_{1}+s_{2})}{(s_{1}^{2}-16)((s_{1}+s_{2})^{2}-16)((2s_{1}+s_{2})^{2}-16)},
\]
which multiplies the matrix with coefficients
\begin{align*}
\text{Coeff$_{1,1}$:}&\qquad(s_{1}+4)(s_{1}+s_{2}+4)(6s_{1}+5s_{2}-12) \\
\text{Coeff$_{1,2}$:}& \qquad 4 (5s_{2}^{2}+13s_{1}s_{2}+9s_{1}^{2}-48)\\
\text{Coeff$_{2,1}$:}& \qquad 	 2 (s_{1}+4)(s_{1}+s_{2}+4)(4 - 6 s_{1} \\& \qquad\quad + 2 s_{1}^2 - 7 s_{2} + 4 s_{1} s_{2} + 2 s_{2}^2)\\
\text{Coeff$_{2,2}$:}&\qquad	192 - 128 s_{1} - 36 s_{1}^2 + 26 s_{1}^3 \\& \qquad\quad- 160 s_{2} - 52 s_{1} s_{2} + 65 s_{1}^2 s_{2} - 20 s_{2}^2\\
\phantom{cc} & \qquad \quad + 
 55 s_{1} s_{2}^2 + 16 s_{2}^3.
 \end{align*}
By combining the above computations we obtain the following asymptotic expansion: 
\begin{align}
\label{E:expansionG}
\bG(s_{1},s_{2}) 
&=
-\frac{1}{4^{5}}
\left(\begin{array}{cc}
 1 & 1/2 \\ 0 & -1/2
\end{array}
 \right) 
 \left(\begin{array}{cc}
 -3\cdot4^{3} & -3\cdot 4^{3} \\ 2\cdot4^{3} & 3\cdot 4^{3}
\end{array}
 \right)
\cdot s_{1}s_{2}(s_{1}+s_{2})
 + o(\|(s_{1},s_{2})\|^{3})
 \notag 
 \\
&=
\frac{1}{32}
\left(\begin{array}{cc}
 4 & 3 \\ 2 & 3
\end{array}
 \right)
 s_{1}s_{2}(s_{1}+s_{2})
 +
 o(\|(s_{1},s_{2})\|^{3}).
\end{align}
This establishes~\eqref{E:leftoprove} and hence concludes the proof of~\eqref{E:mclaurin}. \qed

\section{Interaction of a 1-shock and a 2-shock}
\label{s:12}
We first rigorously state  Lem\-ma~\ref{l:12:heuristic}.
\begin{lemma}
\label{l:12}
There is a sufficiently small constant $\varepsilon >0$ such that if $0 \leq \eta \leq \varepsilon$, then the following holds. 
Assume that the states $U_\ell$, $U_r \in \R^3$ satisfy
 \begin{equation}
 \label{e:primainterazione12}
 U_r = D_1 \Big[\sigma, D_2 [s, U_\ell] \Big]
 \end{equation}
 for real numbers $s$, $\sigma$ such that
$$
 \sigma, s <0, \quad |s|, \; |\sigma| <\frac{1}{4}. 
$$
{Furthermore, assume} that $|U_\ell| <1/2$. 
Then there are real numbers $\sigma'$ and $\tau'$ such that
 \begin{equation}
 \label{e:dopointerazione12}
 U_r = D_3 \Big[ \tau', D_2 \big[ s, D_1 [\sigma', U_\ell] \big] \Big]\,
 \end{equation}
 and 
 \begin{equation}
 \label{e:sigmaprimotau}
 2 \sigma \leq \sigma ' \leq \frac{1}{2} \sigma , \quad 
\frac{1}{100} \sigma s \leq \tau' \leq 10 \sigma s. 
 \end{equation}
\end{lemma}
Note that~\eqref{e:sigmaprimotau} implies $\sigma'<0$ and $\tau'>0$. If we combine these inequalities with~\eqref{e:dopointerazione12} and $s<0$ we see that the three outgoing waves are all shocks. 

Establishing the proof of Lemma~\ref{l:12} amounts to establish~\eqref{e:sigmaprimotau}.
Indeed,
\begin{enumerate}
\item by using Lax's construction (see~\S~\ref{ss:lax}) we determine $\sigma',s',\tau'$ 
such that
\[
U_r = D_3 \Big[ {\tau'}, D_2 \big[ s', D_1 [\sigma', U_\ell] \big] \Big].
\]
\item By combining~\eqref{e:wavefan1},~\eqref{e:wavefan2} and Lemma~\ref{l:wavefan2} we obtain that  $s'=s$.
\end{enumerate}
To establish~\eqref{e:sigmaprimotau} we proceed as follows: 
 \begin{itemize}
 \item[\S~\ref{ss:etazero}:] we establish~\eqref{e:sigmaprimotau} in the case when $\eta =0$. 
\item[\S~\ref{ss:etanonzero}:] we conclude the proof by relying on a perturbation argument. 
More precisely, by using the fact that the flux $F_\eta$ in~\eqref{E:pertSyst} smoothly depends on $\eta$ we show that~\eqref{e:sigmaprimotau} holds provided $\eta$ is sufficiently small.
\end{itemize}
Note that, as we have mentioned in the introduction, the precise expression of the function $p_1$ and $p_3$ plays no role in the proof of Lemma~\ref{l:12}, what is actually relevant 
 is that estimates~\eqref{e:sigmaprimotau} hold at $\eta=0$ with strict inequalities and that $\eta$ is sufficiently small.

\subsection{Proof of Lemma~\ref{l:12}: the case $\eta =0$}
\label{ss:etazero}
We establish~\eqref{e:sigmaprimotau} in the case $\eta=0$. This part of the proof is actually the same as in~\cite[p. 844-845]{BaJ}, but for completeness we go over the main steps. 

We term $\sigma'_0$, $\tau'_0$ the real numbers satisfying~\eqref{e:dopointerazione12} when $\eta =0$. Let $v_r$ and $v_\ell$ denote the second components of $U_r$ and $U_\ell$, respectively.
We term $\gamma$ the speed of the incoming 2-shock (which is the same as the speed of the outgoing 2-shock), {we recall Lemma~\ref{l:wavefan2} and the fact that the second component varies only across 2-shocks. We conclude that}
$$
 \gamma = \frac{v^2_r- v^2_\ell}{v_r - v_\ell} = v_r + v_\ell = 2 v_\ell +s . 
$$ 
Since by assumption $|U_\ell|<1/2$ and $|s|< 1/4$, then 
\begin{equation}
\label{e:bounsugamma}
 |\gamma | < 3. 
\end{equation}
By imposing the Rankine-Hugoniot conditions on the incoming and outgoing 2-shocks and by arguing as in~\cite[pp. 844-845]{BaJ}, with the choice $c=4$, we arrive at the following system:
\begin{equation*}
\left\{
\begin{array}{ll}
(\gamma +4) \sigma_0' + (\gamma -4) \tau'_0 = (\gamma +4) \sigma \\
v_\ell (\gamma +4) { \sigma_0'}+ ({ v_\ell}+s -2 )(\gamma -4) \tau'_0 = (v_\ell +s) (\gamma +4 ) \sigma \\ 
\end{array}
\right.
\end{equation*}
If we set
\begin{equation}
\label{E:linsys}
A: = 
\left(
\begin{array}{cc}
\gamma +4 & \gamma-4 \\
v_\ell (\gamma +4) & (v_\ell +s -2) (\gamma -4 ) \\
\end{array}
\right)
\end{equation}
and 
\begin{equation}
\label{e:icszeroipsilon}
X_0 =
\left(
\begin{array}{cc}
\sigma_0' \\
\tau'_0 \\
\end{array}
\right), \quad 
Y = 
\left(
\begin{array}{cc}
\gamma + 4 \\
(v_\ell +s)(\gamma+4)\\
\end{array}
\right) \sigma,
\end{equation}
then the above linear system can be recast as $A X_0 = Y$.
The explicit expression of the matrix $A^{-1}$ is
\begin{equation}
\label{E:allamenouno}
 \frac{1}{(4^2- \gamma^2)(-s +2)}
\left(
\begin{array}{cc}
(v_\ell +s -2) (\gamma -4 ) & -( \gamma-4) \\
- v_\ell (\gamma +4) & (\gamma + 4 ) \\
\end{array}
\right)
\end{equation}
We solve for $\sigma'_0$ and $\tau'_0$ and we obtain 
\begin{equation}
\label{E:icszero}
 \sigma'_0 = \frac{2 }{-s+2 } \sigma, \qquad 
 \tau'_0 = \frac{\gamma +4}{(4-\gamma)(-s +2)} \, s \sigma \end{equation}
By using~\eqref{e:bounsugamma} and the inequality $|s| < 1/4$, we obtain
\begin{equation}
\label{e:boundazero}
 \frac{2}{3} < \frac{2 }{-s+2 } < 1, \quad \frac{1}{21} < \frac{\gamma +4}{(4-\gamma)(-s +2)} < 4 
\end{equation}
and this implies that the estimate~\eqref{e:sigmaprimotau} holds true in the case when $\eta=0$.

\subsection{Proof of Lemma~\ref{l:12}: the case $\eta >0$}
\label{ss:etanonzero}
We are now ready to complete the proof of Lemma~\ref{l:12}. We proceed as follows:
\begin{itemize}
\item[\S~\ref{sss:preliminary}:] we make some preliminary considerations which reduce the proof of Lemma~\ref{l:12}  to the proof of the fact that a certain map is a strict contraction.
\item[\S~\ref{sss:step2}:] we conclude the proof by showing that the map is indeed a strict contraction. 
\end{itemize}
\subsubsection{Preliminary considerations}
\label{sss:preliminary}
We first introduce some notation. We term $U_m$ the intermediate state \emph{before} the interaction, namely
\begin{equation} 
\label{e:uemme}
 U_m : = D_2[s, U_\ell]. 
\end{equation}
Also, we term $U'_m$ and $U''_m$ the intermediate states \emph{after} the interaction, namely
\begin{equation} 
\label{e:uemmeprimo}
\begin{split}
&  U'_m : = D_1 [\sigma', U_\ell], \\ & U''_m : = D_2 [s, U'_m] = 
{D_3 \big[ -\tau' , U_3 \big] =
 D_3 \big[ -\tau' , D_1[ \sigma, U_m] \big]} \\
 & {\qquad =
 D_3 \big[ -\tau' , D_1[\sigma, D_2 [s, U_\ell] ] \big]} \\
\end{split}
\end{equation}
Next, we use~\cite[eq. (5.3)-(5.4)]{BaJ} and we recast the Rankine-Hugoniot conditions for 2-shocks as 
 a nonlinear system in the form 
 \begin{equation}
 \label{e:nonlinearsystem}
 A X + \eta \mathcal F(X, U_\ell, s, \sigma) = Y,
 \end{equation}
 where $A$ and $Y$ are as in~\eqref{E:linsys} and~\eqref{e:icszeroipsilon}, respectively. Also, the vector $X$ is defined by setting 
 $$
 X: = 
 \left(
 \begin{array}{cc}
 \sigma' \\
 \tau' \\
 \end{array}
 \right)
 $$
 and the nonlinear term $\mathcal F(X, U_\ell, s, \sigma) $ is equal to
\begin{equation}
\label{E:effe}
 \left(
 \begin{array}{cc}
 p_1(U''_m) - p_1(U_m) - p_1(U'_m) + p_1(U_\ell) \\
 p_3(U''_m) - p_3(U_m) - p_3(U'_m) + p_3(U_\ell), \\
 \end{array}
 \right).
\end{equation}
In the above expression, the functions $p_1$ and $p_3$ are the same as in~\eqref{E:pq}. 
{Note, however, that the precise expression of $p_1$ and $p_3$ plays no role in the proof, the only relevant point is that $p_1$ and $p_3$ are both regular (say twice differentiable with Lipschitz continuous second derivatives).}  
Note furthermore that we can regard $\mathcal F$ as a function of $X$, $U_\ell$, $s$ and $\sigma$
 because, owing to~\eqref{e:uemme} and~\eqref{e:uemmeprimo}, $U_m$, $U'_m$ and $U''_m$ are functions of $X$, $U_\ell$, $s$ and $\sigma$. Next, we rewrite equation~\eqref{e:nonlinearsystem} as 
 \begin{equation}
 \label{E:fpp}
 X = 
 X_0 - \eta A^{-1} \mathcal F(X, U_\ell, s, \sigma),
 \end{equation}
 where the vector $X_0 = A^{-1} Y$ is given by~\eqref{e:icszeroipsilon} and~\eqref{E:icszero}. 
 
 We now fix $s$, $\sigma$, $\eta$ and $|U_\ell|$ satisfying the assumptions of Lemma~\ref{l:12} and we define the closed ball 
 \begin{equation}
 \label{e:palla}
 \mathfrak K : = \big\{ X= (\sigma', \tau') \in \R^2: \, |X - X_0| \leq k \eta \sigma s \big\}.
 \end{equation}
 In the above expression, $k>0$ is a universal constant that will be determined in the following and $\sigma_0'$ and $\tau'_0$ are defined by~\eqref{E:icszero}. We also define the function $T: \R^2 \to \R^2$ by setting 
 \begin{equation}
 \label{e:T}
 T(X) : = X_0 - \eta A^{-1} \mathcal F(X, U_\ell, s, \sigma).
 \end{equation}
 Assume that $T$ is a strict contraction from $\mathfrak K$ to $\mathfrak K$. Then the proof of Lemma~\ref{l:12} is complete: indeed, owing to~\eqref{E:fpp} the fixed point $X$ satisfies the Rankine-Hugoniot conditions~\eqref{e:nonlinearsystem}. Also, owing to~\eqref{e:boundazero} and to~\eqref{e:palla} we infer that the inequalities~\eqref{e:sigmaprimotau} are satisfied provided that the parameter $\eta$ is sufficiently small. 
 \subsubsection{Conclusion of the proof of Lemma~\ref{l:12}}
 \label{sss:step2}
 In this paragraph we prove that the map $T$ defined by~\eqref{e:T} is a strict contraction on the closed set $ \mathfrak K $ defined by~\eqref{e:palla}. 
 
{First, we make some remarks about notation. To simplify the exposition, in the following we denote by $C$ a universal constant: its precise value can vary from occurrence to occurrence. Also, in the following we will determine the constant $k$ in~\eqref{e:palla} and then choose the constant $\eta$ in such a way that $k \eta \leq 1$. This choice implies in particular that, when $X$ belongs to the set $\mathfrak K$ defined as in~\eqref{e:palla} and the hypotheses of Lemma~\ref{l:12} are satisfied, then the map $\mathcal F$ attains values on a bounded set and so $\mathcal F$ and all its derivatives are bounded by some constant $C$. Finally, note that, if $X \in \mathfrak K$ and the hypotheses of Lemma~\ref{l:12} are satisfied, then $|A^{-1} | \leq C$.
 
 We now proceed according to the following steps. \\
{\sc Step 1:} we point out that to show that the map $T$ is a contraction it suffices to show that 
\begin{equation}
\label{e:doppiastima}
      |\mathcal F (X, U_\ell, s, \sigma)| \leq C \sigma s
\end{equation}
provided that  $X \in \mathfrak K$ and the hypotheses of Lemma~\ref{l:12} hold. 
Indeed, assume that~\eqref{e:doppiastima} holds, then 
$$
   |T(X) - X_0|   \stackrel{\eqref{e:T}}{\leq} \eta 
  | A^{-1} \mathcal F (X, U_\ell, s, \sigma)|
  \stackrel{\eqref{e:doppiastima}}{\leq} C \eta \sigma s
$$
and hence $T$ attains values in the set $\mathfrak K$ defined as in~\eqref{e:palla} 
provided that $k$ is large enough. Also, 
\begin{equation*}
\begin{split}
    |T(X_1) - T(X_2)| & \stackrel{\eqref{e:T}}{\leq} 
    \eta |A^{-1}| |\mathcal F (X_1, U_\ell, s, \sigma) - 
    \mathcal F (X_2, U_\ell, s, \sigma)| \\
    & \leq \eta C |X_1 - X_2| \leq 
    \frac{1}{2}  |X_1 - X_2|
\end{split}
\end{equation*}
provided that the constant $\eta$ is sufficiently small. This implies that $T$ is a contraction and concludes the proof of Lemma~\ref{l:12}. \\
{\sc Step 2:} we establish~\eqref{e:doppiastima}. First, we point out that, if $X \in \mathfrak K$ and the hypotheses of Lemma~\ref{l:12} are satisfied, then
\begin{equation*}
\begin{split}
    | p_1 (U''_m ) - p_1 (U_m) & - p_1 (U'_m) + p_1 (U_\ell) |
     \leq C  \Big( |  
    U''_m  -  U_m | + | U'_m -  U_\ell  | \Big) \\
    & \quad \leq \quad 
    C  \Big( |  
    U''_m  - U_r | + |U_r- U_m | + | U'_m -  U_\ell  | \Big) \\
    & \quad \leq \quad 
    C \Big( |\tau'| + |\sigma | + | \sigma' | \Big)
    \leq C \Big( |\sigma| + | X_0' | + |X - X_0'| \Big)  
    \\  &  
     \stackrel{\eqref{E:icszero},\eqref{e:palla}}{\leq} 
    C \Big( |\sigma|  + \eta k \sigma s \Big)
    \leq C |\sigma|. 
    \\
\end{split}
\end{equation*}
By using an analogous argument, we control the second component of $\mathcal F$ and we arrive at 
\begin{equation}
\label{e:stimasigma}
   |\mathcal F(X, U_\ell, s, \sigma)|
   \leq C |\sigma|. 
\end{equation}
Next, we point out that when $s=0$ we have $U'_m = U''_m$ and $U_\ell = U_m$ and by using again the Lipschitz continuity of the functions $p_1$ and $p_3$ we conclude that 
\begin{equation}
\label{e:stimaesse}
   |\mathcal F(X, U_\ell, s, \sigma)|
   \leq C |s|. 
\end{equation}
Finally, we use the regularity of the function $\mathcal F$ and, by arguing as in the proof of~\cite[Lemma 2.5, p. 28]{Bre}, we combine~\eqref{e:stimasigma} and~\eqref{e:stimaesse} to obtain~\eqref{e:doppiastima}. This concludes the proof of Lemma~\ref{l:12}. \qed
} 
\section*{Acknowledgments}
{The authors wish to thank the anonymous referees for their useful remarks that helped improve the exposition and allowed to shorten and improve the proof of Lemma~\ref{l:12}.}
The second author wish to thank the organizers of the conference ``Contemporary topics in Conservation Laws" for the invitation to give a talk and present results related to the topic of the present paper.
Both authors are members of the Gruppo Nazionale per l'Analisi Matematica, la Probabilit\`a e le loro Applicazioni (GNAMPA) of the Istituto Nazionale di Alta Matematica (INdAM) and are supported by the PRIN national project ``Nonlinear Hyperbolic Partial Differential E\-qua\-tions, Dispersive and Transport Equations: theoretical and applicative
aspects''. This work was also supported by the EPSRC Science and Innovation award to the OxPDE (EP/E035027/1).


\providecommand{\href}[2]{#2}
\providecommand{\arxiv}[1]{\href{http://arxiv.org/abs/#1}{arXiv:#1}}
\providecommand{\url}[1]{\texttt{#1}}
\providecommand{\urlprefix}{URL }

\end{document}